\documentclass[12pt,reqno]{amsart}
\usepackage{graphicx}
\usepackage{psfrag}
\usepackage{pxfonts}
\usepackage{mathrsfs}
\usepackage{color}
\usepackage{amsmath,amssymb}

\numberwithin{equation}{section}
\newcommand{\diff}{\operatorname{Diff}}

\newcommand{\diam}{\operatorname{diam}}

\theoremstyle{plain}
\newtheorem{maintheorem}{Theorem}

\newcommand{\R}{\mathbb{R}}
\newcommand{\N}{\mathbb{N}}
\newcommand{\Z}{\mathbb{Z}}

\newcommand{\T}{\mathbb{T}}

\newcommand{\vol}{\mathrm{Vol}}
\newcommand{\jac}{\mathrm{Jac}}
\newcommand{\dime}{\mathrm{dim}}

\newtheorem{theorem}{Theorem}[section]

\newtheorem{corollary}[theorem]{Corollary}
\newtheorem{proposition}[theorem]{Proposition}
\newtheorem{lemma}[theorem]{Lemma}
\newtheorem{definition}[theorem]{Definition}

\theoremstyle{remark}
\newtheorem{remark}[theorem]{Remark}

\begin{document}

\thanks{}

\author{J. Santana C Costa}
\address{DEMAT-UFMA S\~{a}o Lu\'{i}s-SP, Brazil.}
\email{jsc.costa@ufma.br}

\author{F. Micena}
\address{
  IMC-UNIFEI Itajub\'{a}-MG, Brazil.}
\email{fpmicena@gmail.com}


\renewcommand{\subjclassname}{\textup{2000} Mathematics Subject Classification}

\date{\today}

\setcounter{tocdepth}{2}

\title{Pathological center foliation with dimension greater than one}
\maketitle
\begin{abstract}
 In this paper we are considering partially hyperbolic diffeomorphims of the torus, with $dim(E^c) > 1.$ We prove, under some conditions, that if the all center Lyapunov exponents of the linearization $A,$ of  a \mbox{DA-diffeomorphism} $f,$ are positive and the center foliation of $f$ is absolutely continuous, then  the sum of the center Lyapunov exponents of $f$ is bounded by the sum of the center Lyapunov exponents of $A.$ After, we construct a $C^1-$open class of volume preserving \mbox{DA-diffeomorphisms}, far from Anosov diffeomorphisms, with non compact pathological two dimensional center foliation. Indeed, each $f$ in this open set satisfies the  previously established hypothesis, but the sum of the center Lyapunov exponents of $f$ is greater than the corresponding sum with respect to  its linearization. It allows to conclude that the center foliation of $f$ is non absolutely continuous. We still build an example of a DA-diffeomorphism, such that the disintegration of volume along the two dimensional, non compact center foliation is neither Lebesgue nor atomic.
\end{abstract}

\section{Introduction}

A diffeomorphism $f:M\rightarrow M$ of a compact closed smooth manifold is partially hyperbolic if
the tangent bundle splits into three invariant sub bundles $TM=E^s\oplus E^c\oplus E^u$ such that $E^s$ is contracting, $E^u$ is expanding and $E^c$ has an
 intermediate behavior, that is, not as contracting as $E^s$ and nor as expanding as $E^u$. If $E^c=\{0\}$, then $f$ is called
 uniformly hyperbolic or Anosov diffeomorphism. We denote by $PH^r_{\omega}(M),$ the set of all $C^r-$partially hyperbolic diffeomorphism preserving the $\omega-$form. Here we study partially hyperbolic diffeomorphisms on torus ${\T}^d$ homotopic
to a linear Anosov which are known as Derived from Anosov (DA). By \cite{HPS} and \cite{Brpe} for partially hyperbolic diffeomorphisms,
there is foliations $W^s$ and $W^u$ tangent to $E^s$ and $E^u$, respectively, but the distribution $E^c$ may not be integrable, for instance see section 6.1 of \cite{pesin2004lectures}.
If $E^c$ is  one dimensional, then it is integrable, but not necessarily uniquely integrable (see \cite{HHU2}). In \cite{B} shows that  for (absolute) partially hyperbolic diffeomorphisms if the foliations $W^s$ and $W^u$ has a geometrical condition (quasi-isometric), then $E^c$ is uniquely integrable, that is, there
is a foliation $W^c$ tangent to $E^c$. Our results relate absolute continuity  of the center foliation $W^c$ and the Lyapunov exponents.

The Lyapunov exponents play an important role in the ergodic theory and dynamical systems. They are useful tool of the Pesin theory, in the study of
entropy, equilibrium states among others. The existence of these exponents is guaranteed by celebrated Osceledec's Theorem \cite{OS}.
In general the Lyapunov exponents not vary continuously with $x\in M$ or with the dynamics in the ambient
$\diff^1(M)$.
We show, under some conditions, that the partially hyperbolic diffeomorphsms in the same homotopic class has its center Lyapunov exponents bounded by center exponents of the its linearization, more precisely:

\begin{maintheorem}\label{Teo E}
Let  $f:{\T}^d\rightarrow {\T}^d$ be a volume preserving  DA diffeomorphism and consider $A$ its linearization  such that:
\begin{enumerate}
  \item $\dime E^c_f =\dime E^c_A=d_c.$
  \item $E^c_A=E^c_1\oplus E^c_2\oplus\cdots \oplus E^c_{d_c},$ such that each $E^c_i, i =1, \ldots, d_c$ is a unidimensional eigenspace of $A$ with the corresponding Lyapunov exponent, $\lambda_i^c(A)>0,$ for any $i = 1, 2, \ldots, d_c.$
  \item $f$ is dynamically coherent and in the universal cover $\angle(T_x\widetilde{W}^c_f, (\widetilde{E}^c_A)^{\bot})>\alpha>0$ for any $x\in{\R}^d.$
\end{enumerate}
If the center foliation $W^c_f$ is absolutely continuous, then
 $\displaystyle\sum_{i=1}^{d_c}\lambda_i^c(f,x)\leq\displaystyle \sum_{i=1}^{d_c}\lambda_i^c(A)$  for $m-$ almost everywhere
 $x\in {\T}^d$.
\end{maintheorem}

Absolutely continuous of foliations is a very related property with Lyapunov exponents. We understand it by absolutely continuous the behavior of
foliation relative to disintegration of the volume in the foliation. By Fubini's Theorem we know
that  $C^1$-foliations are absolutely continuous, but this property is not valid for $C^0$-foliation.
The non absolutely continuous foliation have
been referred to as "Fubini's Nightmare" or "Fubini Foiled", Katok construct  examples of foliation with this property (see \cite{Mi}),
for more definitions and discussions
about absolutely continuous see \cite{PVW}. In  \cite{A}, \cite{AS} it is shown that the stable and unstable
foliations of $C^2-$Anosov diffeomorphisms are absolutely continuous and this property is fundamental in the
proof of the ergodicity of the volume preserving $C^2-$Anosov diffeomorphisms, \cite{A}. We also know that for partially hyperbolic $C^2$ diffeomorphisms the stable and unstable
foliations are absolutely continuous \cite{Brpe}, but in general, we do not know if the center foliation
(when there is) is absolutely continuous. One of the first to study the behavior of the center foliation
was R. Ma\~n\'e, in a letter (unpublished) to M. Shub,
they relate the absolutely continuous of compact center foliations in which the Lyapunov exponents are non zero.
These ideas were very useful in the study on absolutely continuous of compact center foliations
(see \cite{RW} and \cite{HP}). The non absolutely
continuous of the non compact center foliations is also very common, in \cite{Go} show there are open sets in $PH^r_m({\T}^3), r \geq 2,$
of diffeomorphisms with one dimensional non compact center foliation and non absolutely continuous. We build an open set $U\subset PH^r_m({\T}^4),$ \mbox{$r \geq 2$},
of diffeomorphisms with non compact two dimensional center foliation and non absolutely continuous.

Denote by $DA^r_m(\mathbb{T}^d)$ the set of all  $C^r$ and $m$ preserving DA diffeomorphism of $\mathbb{T}^d$ and $\mathcal{A}(\mathbb{T}^d)$
the set of all Anosov diffeomorphism of $\mathbb{T}^d.$


\begin{maintheorem} There is a $C^1-$open set $U \subset DA^r_m(\mathbb{T}^4) \setminus \overline{\mathcal{A}(\mathbb{T}^4)}, $
with $r \geq 2,$ such that each $f \in U$ has the same diagonalizable
linearization $A,$ $\dim E^c_f = 2,$  satisfying the hypothesis $1), 2) $ and $3) $  of Theorem \ref{Teo E}, but
$\sum_{i = 1}^{2} \lambda^c_i(f, x) >  \sum_{i = 1}^{2} \lambda^c_i(A),$  for $m$ almost everywhere
$x \in \mathbb{T}^4.$ Particularly $ W^c_f$ is non absolutely continuous for every $f \in U.$
\end{maintheorem}

This Theorem is related to a result of \cite{SX} which shows the existence of open sets in $PH^r_m({\T}^d)$ of diffeomorphisms near a linear automorphism with non absolutely continuous one dimensional center foliation. In \cite{M16}, the second autor display a open set in $PH^r_m({\T}^3)$ of diffeomorphisms far from linear Anosov with non absolutely continuous one dimensional center foliation. Our results are a generalized version of results of \cite{M16}.

A natural question is how is disintegration of the volume form along the foliations in these cases? This is, how are conditional
measures of non absolutely continuous foliations?
In  \cite{PT} and \cite{PTV} built examples of partially hyperbolic diffeomorphisms of ${\T}^3$ wherein
the disintegration of the center foliation is atomic (totally contrary to the Lebesgue measure).
In \cite{AVW} it is shown for perturbation of time-one of geodesic flow, the disintegration along the center foliation
is or Lebesgue or atomic. In \cite{V} and \cite{YV} it is shown which the disintegration of the one dimensional center foliation
of partially hyperbolic diffeomorphisms of ${\T}^3$ can be neither Lebesgue nor atomic (in contrast to the dichotomy of \cite{AVW}).
We construct an example of the partially hyperbolic diffeomorphism with two dimensional center foliation whose disintegration
is neither Lebesgue nor atomic.

\begin{maintheorem}\label{thmB}
There is a DA diffeomorphism $f:{\T}^4\rightarrow{\T}^4$ volume preserving with non compact two dimensional center foliation such that the disintegration of volume along the center leaves is neither Lebesgue nor atomic.
\end{maintheorem}

\section{Preliminaries}

Let $M$ be a $C^{\infty}$ Riemannian closed (compact, connected and boundaryless)  manifold.
A $C^1-$diffeomorphism $f: M \rightarrow M$ is called a partially hyperbolic diffeomorphism if the tangent bundle $TM$ admits a $Df$ invariant tangent decomposition $TM =  E^s \oplus E^c \oplus E^u$ such that all unitary vectors $v^{\sigma} \in E^{\sigma }_x, \sigma \in \{s,c,u\}$ for every $x \in M$ satisfy:

$$ ||D_x f v^s || < ||D_x f v^c || < ||D_x f v^u ||,$$
moreover

$$||D_x f v^s || < 1 \;\mbox{and}\; ||D_x f v^u || > 1 $$

We say that a partially hyperbolic diffeomorphism $f$ is an absolute partially hyperbolic diffeomorphism if

$$||D_x f v^s || < ||D_y f v^c || < ||D_z f v^u ||$$
for every $x,y,z \in M$ and $v^s, v^c, v^u$ are unitary vectors in $E^s_x, E^c_y, E^u_z$ respectively.

From now on, in this paper, when we require partial hyperbolicity, we mean absolute partially hyperbolicity and
all diffeomorphisms considered are at least $C^1.$


\begin{definition} A partially hyperbolic diffeomorphism $f: M \rightarrow M$ is called dynamically coherent if $E^{cs} := E^c \oplus E^s$
and $E^{cu} := E^c \oplus E^u$  are uniquely integrable to invariant foliations $ W^{cs}$ and  $ W^{cu},$ respectively the
center stable and center unstable foliations. Particularly $E^c$ is uniquely integrable to the center foliation $ W^{c},$
which is obtained by the intersection $ W^{cs} \cap  W^{cu}.$
\end{definition}


%

Every diffeomorphism of the torus $\mathbb{T}^d$ induces an automorphism of the fundamental group and there exists a unique linear
diffeomorphism $f_{\ast}$ which induces the same automorphism on $\pi_1(\mathbb{T}^d).$ The diffeomorphism $f_{\ast}$ is called linearization of $f.$

\begin{definition} Let $f: \mathbb{T}^d \rightarrow \mathbb{T}^d $ be a partially hyperbolic diffeomorphism, $f$ is called a
derived from Anosov (DA) diffeomorphism if its linearization $f_{\ast}: \mathbb{T}^d \rightarrow \mathbb{T}^d $ is a linear Anosov automorphism.
\end{definition}

\begin{theorem}[\cite{franks1969anosov}, \cite{manning1974there}]
Let $f: \mathbb{T}^d \rightarrow \mathbb{T}^d $ be a derived from Anosov diffeomorphism, then $f$ is  semi-conjugate to $A$, that is,
there exists a continuous surjective
function $h:{\T}^d\rightarrow{\T}^d$ such that $h\circ f=A\circ h$.
\end{theorem}



\subsection{Lyapunov Exponents}

Lyapunov exponents are important constants and measure the asymptotic behavior of dynamics in tangent space level. Let $f: M \rightarrow M$ be a measure preserving diffeomorphism. Then by  Oseledec's Theorem, for almost every $x \in M $ and any $v \in T_x M $ the following limit exists:

$$\lim_{n \rightarrow +\infty} \frac{1}{n} \log ||Df^n(x) \cdot v ||$$
and it is equal to one of the Lyapunov exponents of $f.$
%
%
%

%
%

\subsection{Absolute Continuity}
Roughly speaking a foliation $ W$ of $M$ is absolutely continuous if satisfies: Given a set $Z \subset M,$ such that $Z$ intersects the leaf $ W(x)$ on a zero measure  set of the leaf, with $x$ along a full Lebesgue set of $M,$ then $Z$ is a zero measure set of $M.$
More precisely we write:

\begin{definition} We say that a foliation $ W$ of $M$ is absolutely continuous if given any $ W-$foliated box $B$ and a Lebesgue measurable set $Z,$ such that $Leb_{ W(x)\cap B} ( W(x)\cap Z) = 0,$ for $m_B-$ almost everywhere $x \in B,$ then $m_B(Z) = 0.$ Here $m_B $ denotes the Lebesgue measure on $B$ and $Leb_{ W(x)\cap B} $ is the Lebesgue measure of the submanifold $ W(x)$ restricted to $B.$
\end{definition}

It means that if $P$ is such that $m_B(P) > 0,$ then there are a measurable subset $B' \subset B,$ such that $m_B(B') > 0$ and  $Leb_{ W(x)\cap B}( W(x) \cap Z) > 0$ for every $x \in B'.$

The study of absolute continuity of the center foliation started with Ma\~{n}\'{e},  that noted a interesting relation between absolute continuity and the center Lyapunov exponent.
The Ma\~{n}\'{e}'s argument can be explained as the following theorem:

\begin{theorem} Let $f: M \rightarrow M$ be a  partially hyperbolic, dynamically coherent such that $\dim(E^c) = 1$ and $ W^c$ is a compact foliation. Suppose $f$ preserves a volume form $m$ on $M,$  and the set of $x \in M$ such that   $\lambda^c_f(x) > 0$ has positive volume.  Then $ W^c$ is non absolutely continuous.
\end{theorem}

\begin{proof} Denote by $P$ the set of $x \in M$ such that   $\lambda^c_f(x) > 0.$ Consider the set
\mbox{$\Lambda_{k,l, n } =\{ x \in P | \; ||Df^j(x)|E^c|| \geq e^{\frac{j}{k}},\; \mbox{for every} \;  j \geq l , \; \mbox{and} \; | W^c(x)| < n   \} ,$} here $| W^c(x)|$ denotes the size of the center leaf $ W^c(x)$ through $x. $

We have $P = \displaystyle\bigcup_{k,l, n \in \mathbb{N}} \Lambda_{k,l, n }, $ in particular there are $k_0, l_0, n_0$ such that \mbox{$m(\Lambda_{k_0,l_0, n_0 }) > 0.$}  Supposing  that $ W^c$ is an absolutely continuous foliation, there is a center leaf $ W^c(x), $ such that it intersects  $\Lambda_{k_0,l_0, n_0 }$ on a positive Lebesgue measure set of the leaf. By Poincar\'{e}-recurrence Theorem, the  point $x $ can be chosen a recurrent point, particularly there is a subsequence $n_k$ such that $f^{n_k}(x) \in \Lambda_{k_0,l_0, n_0 }, $ and it implies that the size $| W^c({f^{n_k}(x)})| < n_0.$

On the other hand, we denote
$\alpha = Leb_{ W^c(x)}( W^c(x) \cap \Lambda_{k_0,l_0, n_0 } ),$
so if $j \geq k_0$ we have $| W^c(f^j(x))| \geq \alpha\cdot e^{\frac{j}{k_0}} \rightarrow +\infty$ when $j \rightarrow +\infty, $ and it contradicts $| W^c({f^{n_k}(x)})| < n_0$ for a subsequence $n_k.$
\end{proof}

Consequently all one dimensional compact and absolutely continuous center foliation implies that $\lambda^c_f(x) = 0,$ for $m-$ almost everywhere $x \in M.$

We make a generalized version of the Ma\~n\'e's argument, comparing volumes, in the proof of the Theorem \ref{Teo E}.

\begin{remark} Katok exhibits an example of a volume preserving partially hyperbolic diffeomorphism $f: \mathbb{T}^3 \rightarrow \mathbb{T}^3 $ such that $ W^c$ is compact, non absolutely continuous and $\lambda^c_f(x)=0$ for $m-$ a.e. $x \in \mathbb{T}^3 . $ See \cite{HaPe} and citations therein.
\end{remark}

\subsection{Quasi-isometry}

An important tool that we need is quasi-isometry of foliations.

\begin{definition}
A foliation $W$ of a closed manifold $M$ is called quasi-isometric if there is a
constant $Q > 0$, such that in the universal cover $\widetilde{M}$ we have:
$$
d_{\widetilde{W}}(x,y)\leq Q d_{\widetilde{M}}(x,y) +Q
$$
for every $x, y$ points in the same lifted leaf $\widetilde{W}$,
where  $\widetilde{W}$ denotes the lift of $W$ on $\widetilde{M}$.
\end{definition}

Here $d_{\widetilde{W}}$ denotes the Riemannian metric on $\widetilde{W}$ and $d_{\widetilde{M}}$ is a
Riemannian metric of the ambient $\widetilde{M}$.

Every diffeomorphism of the torus ${\T}^d$ induces an automorphism of the fundamental group
and there exists a unique linear diffeomorphism $f_{\ast}$ which induces the same automorphism
on $\pi_1({\T}^d)$. The diffeomorphism $f_{\ast}$ is called linearization of $f$. When $f$ is DA diffeomorphism, then its
linearization is the linear Anosov automorphism. Below we have a well-known result.

\begin{proposition}\label{Prop Hamm 1}
	Let $f:{\T}^d\rightarrow {\T}^d$ be a partially hyperbolic diffeomorphism with linearization $A$,
	then for each $n\in{\Z}$ and $C>1$ there is an $M>0$ such that  for all $x,y\in {\R}^d$ and
	$$
	||x-y||>M\,\,\, \Rightarrow\,\,\,\frac{1}{C}<\frac{||\tilde{f}^n(x)-\tilde{f}^n(y)||}{||\tilde{A}^n(x)-\tilde{A}^n(y)||}<C.
	$$.
\end{proposition}

For a subset $X\subset{\R}^d$ and $R>0$, led $B_R(X)$ denote the neighbourhood
$$
B_R(X)=\{y\in{\R}^d;\,\,||x-y||<R\,\,\mbox{for\,\,some}\,\,x\in X\}.
$$
\begin{proposition}[\cite{H}]\label{Prop Hamm 3}
Let $f:{\T}^d\rightarrow {\T}^d$ be a partially hyperbolic diffeomorphism dynamically coherent with linearization $A$,
then there is a constant $R_c$ such that for all $x\in{\R}^d$,
\begin{itemize}
   \item $\widetilde{W}^{cs}_f(x)\subset B_{R_c}(\widetilde{W}^{cs}_A(x))$,
   \item $\widetilde{W}^{cu}_f(x)\subset B_{R_c}(\widetilde{W}^{cu}_A(x))$,
   \item $\widetilde{W}^{c}_f(x)\subset B_{R_c}(\widetilde{W}^{c}_A(x))$.
\end{itemize}

\end{proposition}

\begin{corollary}[\cite{H}]\label{Cor Hamm 2}
 If $||x-y||\rightarrow\infty$ where $y\in \widetilde{W}^{c}_f(x)$
	then $\frac{x-y}{||x-y||}\rightarrow \widetilde{E}^{c}_A(x)$ uniformly.
		More precisely, for $\varepsilon>0$ there exists $M>0$ such that if $x\in{\R}^d$, $y\in \widetilde{W}^c_f(x)$ and $||x-y||>M$, then
	$$
	||\pi_A^{c\perp}(x-y)||<\varepsilon||\pi^{c}_A(x-y)||.
	$$
	where $\pi_A^{c}$ is the orthogonal projection in the subspace $E^c_A$  and $\pi_A^{c\perp}$ is the projection in the orthogonal
	subspace $E^{c\perp}_A.$
\end{corollary}

From Brin, \cite{B} we have.

\begin{proposition}[\cite{B}]\label{B1} Let $W$ be a $k-$dimensional foliation of $\mathbb{R}^d.$ Suppose that there is an $(d - k)$ dimensional plane $P$ such that $T_x W(x)\cap P = \{0\}$ and $\angle(T_x W(x), P) > \beta > 0$ for every $x \in \mathbb{R}^d.$ Then $W$ is quasi-isometric.
\end{proposition}

\begin{theorem}[\cite{B}]\label{B2} Let $f: M \rightarrow M$ a partially hyperbolic diffeomorphism. If $ W^s_f$ and $ W^u_f$ are quasi-isometric foliations, then $f$ is dynamically coherent.
\end{theorem}

\section{Proof of Theorem A}

\begin{lemma}
Let  $f$ be as in the Theorem A, then the foliation $\widetilde{W}^c_f$ is quasi-isometric.
\end{lemma}

\begin{proof}
We just need to apply the Proposition \ref{B1}, doing $P=(\widetilde{E}^c_A)^{\bot}.$ By the third item of the assumptions of Theorem A, we have
$\angle(T_x\widetilde{W}^c_f,P)>\alpha>0$, thus by the Proposition \ref{B1}, $\widetilde{W}^c_f$ is
quasi-isometric.
\end{proof}

\begin{proposition}\label{Prop vol center}
Let   $f:{\T}^d\rightarrow {\T}^d$ and $A$ be as in the Theorem A,

then given  $\varepsilon>0$, there is  $M > 0$ such that
$$
\vol(\tilde{f}^n\pi^{-1}(R))\leq C_0(1+\varepsilon)^{nd_c}e^{n\sum\lambda_i^c(A)}\vol(R)
$$
 for any $n,$  where $\pi$ is the  orthogonal projection from  $\widetilde{W}^c_f(z)$ to $\widetilde{E}^c_A(z)$ (parallel to $(\widetilde{E}^c_A)^{\bot}$) and  $R$ is a hypercube in  $\widetilde{E}^c_A(z)$ contained $z,$ with dimension equal to  $d$ and the length of each edge is bigger than  $M.$
\end{proposition}

To prove the above proposition we need the following auxiliar results:

\begin{lemma}\label{lema Lp}
Let  $f:M\rightarrow M$ be a Lipschitzian map, then there is $K > 0$ such that $|\jac f(x)|\leq K,$ for any $x\in M$ such that  $f$ is differentiable at $x.$
\end{lemma}

\begin{proof}[Proof of Lemma]
Let $L$ be the the Lipschitz constant of $f$, give  $x\in M$ such that $f$ is differentiable and  $v\in T_xM$, then
$$
||D_xf(v)||=\displaystyle\lim_{t\rightarrow 0}\left\|\frac{f(x + tv)-f(x)}{t}\right\|\leq\lim_{t\rightarrow 0}\frac{L||x+tv-x||}{|t|}=L||v||
$$
It implies that  $||D_xf||=\displaystyle\sup_{v\in T_xM}\frac{||D_xf(v)||}{||v||}\leq L.$

Since $\R^{n^2}$  is isomorphic to $M_n(\R),$ the space of all $n\times n-$matrixes with real coefficients,
by equivalence between norms of $\R^{n^2}$ we have \mbox{$\{A \in M_n(\R)|\; ||A||\leq L\}$} is compact.
Since $\det: M_n(\R) \rightarrow \R $ is continuous, then there is $K \geq 0$ such that
$$
|\jac f(x)|=|\det D_xf| \leq K.
$$
It concludes the proof.
\end{proof}

\begin{proposition}\label{Prop recob}
Let $f:X\rightarrow Y$ be a covering map such that $X$ is connected by path and $Y$ is simply connected, then $f$ is a homeomorphism.
\end{proposition}

\begin{proof}[Proof of the Proposition \ref{Prop vol center}]

\textbf{Claim  1:} For  $z\in{\R}^d$, the orthogonal projection \mbox{$\pi: \widetilde{W}^c_f(z)\rightarrow \widetilde{E}^c_A$}
(parallel to  $(\widetilde{E}^c_A)^{\bot}$) is a uniform bi-Lipschitz diffeomorphism.
\begin{proof}[Proof of Claim 1]
The argument used here is similar to that used in \cite{B}, Proposition 4.  By item 3 of the Theorem \ref{Teo E} the plane $(\widetilde{E}_A^c)^{\bot}$ is transversal to
the foliation $\widetilde{W}^c_f$ then there is  $\beta>0$ such that
		\begin{equation}\label{eq10}
		||d\pi(y)(v)||\geq\beta||v||
		\end{equation}
for any $x\in \widetilde{W}^c_f(z)$ and $v\in T_x\widetilde{W}^c_f(z)$, it implies that $\jac \pi(x)\neq 0$,
by Inverse Function Theorem  for each $x\in \widetilde{W}^c_f(z)$ there is a ball
$B(\delta,x)\subset \widetilde{W}^c_f(z)$ such that $\pi|_{B(\delta,x)}$ is a diffeomorphism. From proof of Inverse
Function Theorem and Equation \ref{eq10}, $\delta$ can be taken independent of $x$. Again, by Equation \ref{eq10}
there is $\varepsilon>0$, independent of $x$ such that $B(\varepsilon,\pi(x))\subset \pi(B(\delta,x))$. For to prove which
 $\pi$ is  surjective, we will show that $\pi(\widetilde{W}^c_f(z))$ is open and closed set. As $\pi$ is a local homeomorphism, then
 $\pi(\widetilde{W}^c_f(z))$ is open. To verify that it is also closed let
$y_n\in \pi(\widetilde{W}^c_f(z))$ be a sequence converging to $y$, hence there is a $n_0$ enough large
such that $y\in B(\varepsilon,y_{n_0})$ and therefore $y\in \pi(\widetilde{W}^c_f(z))$, then $\pi$ is surjective.
Moreover $\pi$ is a covering map, in fact for any $y\in \widetilde{E}^c_A$ there is a neighborhood
$B(\varepsilon,y)$ with $\pi^{-1}(B(\varepsilon,y))=\bigcup U_i$, where $\pi:U_i\rightarrow B(\varepsilon,y)$ is
a diffeomorphism. The injectivity follows of the Proposition \ref{Prop recob}.

The map $\pi$ is Lipschitz.  In fact $||\pi(x)-\pi(y)||\leq||x-y||\leq d_{\widetilde{W}_f^c}(x,y)$. Let us show that
$\pi^{-1}$ is  Lipschitz.
From the Equation \ref{eq10},  consider $L=\frac{1}{\beta}$ such that  $||d\pi^{-1}(y)(v)||\leq L||v||$ for any
$y\in W^c_f(z)$ e $v\in T_yW^c_f(z)$. Let  $[x,y]$ be the rect segment in  $\widetilde{E}^c_A$ connecting $x$ to  $y,$ so the set
$\pi^{-1}([x,y])=\gamma$ is a smooth curve connecting the points $\pi^{-1}(x)$ and $\pi^{-1}(y)$ in $\widetilde{W}^c_f(z)$. So,
$$
d_{\widetilde{W}^c}(\pi^{-1}(x),\pi^{-1}(y))\leq {length}(\gamma)=\int_{[x,y]}|d\pi^{-1}(t)|dt\leq L||x-y||.
$$
\end{proof}

As $E^c(A)=E^c_1\oplus E^c_2\oplus\cdots\oplus E^c_{d_c}$, suppose that the edges of the  $d_c$-cube $R$ lies in hypercubes parallel to  $E^c_i.$ \\

As in  (Lemma 3.6, \cite{MT}), we can prove:

\textbf{Claim 2:} Let   $x=(x_1,\ldots,x_i,\ldots,x_{d_c})$ and
$y=(x_1,\ldots,y_i,\ldots,x_{d_c})$ with  $x,y\in\partial R$, $x - y \in E_i^c,$ then given    $n$ and  $\varepsilon>0$, there is $M>0$ such that, if $||x-y||>M,$ then
$$
(1-\varepsilon)e^{n\lambda_i^c(A)}||\pi^{-1}(x)-\pi^{-1}(y)||\leq||\tilde{A}^n\pi^{-1}(x)-\tilde{A}^n\pi^{-1}(y)||
\leq(1+\varepsilon)e^{n\lambda_i^c(A)}||\pi^{-1}(x)-\pi^{-1}(y)||
$$
for all  $i=1,2,\ldots,d_c$

\begin{figure}[!htb]
\centering
\includegraphics[scale=1.1]{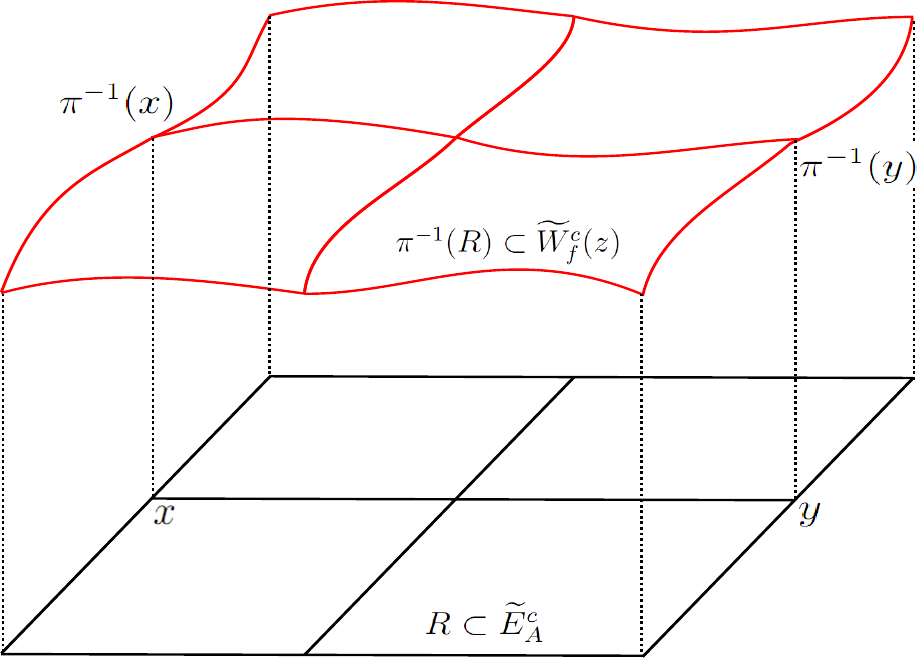}
\caption{}\label{cone1}
\end{figure}

\begin{proof}[Proof of Claim 2]
 Let $\widetilde{E}^c_i(A)$ be the eigenspace corresponding to  $\lambda_i^c(A)$ and  denote by  $\mu:=e^{\lambda_i^c(A)},$ the absolute value of the an eigenvalue $\tilde{A}$ corresponding $\widetilde{E}^c_i(A)$. By Corollary \ref{Cor Hamm 2} we have
$$
\frac{\pi^{-1}(x)-\pi^{-1}(y)}{||\pi^{-1}(x)-\pi^{-1}(y)||}=v_i^c+e_M
$$
where  $v_i^c$ is an  unit eigenvector of $\tilde{A}$ in $\widetilde{E}^c_i(A)$ direction  and  $e_M$ is a vector correction,
converges uniformly  to  zero when  $M \rightarrow +\infty.$ It follows that
$$
\tilde{A}^n\left(\frac{\pi^{-1}(x)-\pi^{-1}(y)}{||\pi^{-1}(x)-\pi^{-1}(y)||}\right)=\mu^n v_i^c+\tilde{A}^ne_M=
\mu^n\left(\frac{\pi^{-1}(x)-\pi^{-1}(y)}{||\pi^{-1}(x)-\pi^{-1}(y)||}\right)-\mu^n e_M + \tilde{A}^ne_M
$$
it implies that
$$
\frac{||\tilde{A}^n(\pi^{-1}(x)-\pi^{-1}(y))||}{||\pi^{-1}(x)-\pi^{-1}(y)||}=\left\|\mu^n\left(\frac{\pi^{-1}(x)-\pi^{-1}(y)}{||\pi^{-1}(x)-\pi^{-1}(y)||}\right)
-\mu^n e_M +\tilde{A}^ne_M\right\|
$$
thus
\begin{align*}
||\pi^{-1}(x)-\pi^{-1}(y)||&(\mu^n-\mu^n||e_M||-||\tilde{A}^n||||e_M||)\leq||\tilde{A}^n(\pi^{-1}(x)-\pi^{-1}(y))||\\
                           &\leq||\pi^{-1}(x)-\pi^{-1}(y)||(\mu^n+\mu^n||e_M||+||\tilde{A}^n||||e_M||)
\end{align*}
Since  $n$ is fixed and  $||e_M||\rightarrow 0$ when $M\rightarrow\infty$, choose a large  $M$ such that
$$
\mu^n||e_M||+||\tilde{A}^n||||e_M||\leq\varepsilon\mu^n
$$
we conclude that
$$
(1-\varepsilon)e^{n\lambda_i^c(A)}||\pi^{-1}(x)-\pi^{-1}(y)||\leq||\tilde{A}^n\pi^{-1}(x)-\tilde{A}^n\pi^{-1}(y)||
\leq(1+\varepsilon)e^{n\lambda_i^c(A)}||\pi^{-1}(x)-\pi^{-1}(y)||
$$
\end{proof}

\textbf{Claim 3:}  There is $M > 0,$ such that $\tilde{f}^n(\pi^{-1}(R)) \subset \pi^{-1}((1+ \varepsilon)^n\tilde{A}^n(R)),$ for every $n \geq 1,$ up to translation of $(1+ \varepsilon)^n\tilde{A}^n(R).$  We are considering $R$  a $d_c$ hypercube with edges parallel to each direction $E^c_i, i =1, \ldots, d_c,$ (see Figure 1), and the length of each edge of $R$ is at least $M.$

\begin{proof}

Since $||\tilde{f} - \tilde{A}|| < K,$ for some $K > 0,$ We have $$\tilde{A}(\pi^{-1}(R)) \subset B_K(\tilde{f}(\pi^{-1}(R)).$$
Moreover, $\tilde{A}(\pi^{-1}(R))$ projects on  $\tilde{A}(R),$ via $\pi.$
We can choose $M,$ very large, such that $\tilde{f}(\pi^{-1}(R)) \subset \pi^{-1}((1+ \varepsilon)\tilde{A}(R)),$ where $(1+ \varepsilon)\tilde{A}(R)$ denotes a dilatation of $\tilde{A}(R),$ by a factor $1+ \varepsilon.$  For this,  is suffices that
$$\lambda M (1+ \varepsilon ) > \lambda M + 2K,$$ for every $\lambda$ a center eigenvalue of $A.$ So it is sufficient that $\frac{2K}{\lambda M} < \varepsilon.$

Thus, for $n = 1, $ we have
\begin{equation}
\tilde{f}(\pi^{-1}(R)) \subset \pi^{-1}((1+ \varepsilon)\tilde{A}(R)),
\end{equation}
 up to a possible translation of $\tilde{A}(R).$

Now, denote $R_1 = (1+ \varepsilon)\tilde{A}(R) .$  It is important to note that
\mbox{$||\tilde{f}\pi^{-1}(x)-\tilde{f}\pi^{-1}(y)||>M.$} In fact, we can take $M$ satisfying also the Proposition \ref{Prop Hamm 1} and  Claim 2, so

\begin{align*}
||\tilde{f}\pi^{-1}(x)-\tilde{f}\pi^{-1}(y)||&\geq (1+\varepsilon)^{-1}||
\tilde{A}\pi^{-1}(x)-\tilde{A}\pi^{-1}(y)||\\
&\geq (1+\varepsilon)^{-1}(1-\varepsilon)e^{\lambda^c_i(A)}||\pi^{-1}(x)-\pi^{-1}(y)||\\
&\geq ||x - y|| \geq M,
\end{align*}

if $\varepsilon$ is enough small.

So we can do the same process before, for $R_1,$ we obtain:
$$\tilde{f}(\pi^{-1}(R_1)) \subset \pi^{-1}((1+ \varepsilon)\tilde{A}(R_1)),$$

Since $\tilde{f}^2(\pi^{-1}(R)) \subset \tilde{f}(\pi^{-1}(R_1)),$ it implies
$$\tilde{f}^2(\pi^{-1}(R)) \subset \pi^{-1}((1+ \varepsilon)^2\tilde{A}^2(R)). $$

Following this inductive argument, we conclude that

\begin{equation}\label{contem}
\tilde{f}^n(\pi^{-1}(R)) \subset \pi^{-1}((1+ \varepsilon)^n\tilde{A}^n(R)),
\end{equation}
for every $n \geq 1.$ \end{proof}

By Lemma \ref{lema Lp} follows that
\begin{align*}
\vol(\tilde{f}^n\pi^{-1}(R))&\leq\vol(\pi^{-1}((1+\varepsilon)^{2n}\tilde{A}^n(R))= C_0(1+ \varepsilon)^{nd_c} e^{n \sum_{i = 1}^{d_c} \lambda^c_i(A)}\vol(R),
\end{align*}
The constant  $C_0,$ comes from Lemma \ref{lema Lp}. It concludes the proof of the Proposition \ref{Prop vol center}.
\end{proof}

\begin{proof}[Proof of Theorem \ref{Teo E}]

Suppose by contradiction, $W^c_f$ is absolutely continuous and there exists a positive volume set  $Z\in {\T}^d$ such that
$\sum_i\lambda_i^c(f,x)>\sum_i\lambda_i^c(A)$ for any $x\in Z$.

Let $P:{\R}^d\rightarrow {\T}^d$ be the covering map, $D\subset{\R}^d$ a fundamental domain  and $\widetilde{Z}=P^{-1}(Z)\cap D$.  We have $\vol(\widetilde{Z})>0$. Since $D\tilde{f}^n$ and $Df^n$ are conjugate matrices, $ D\tilde{f}^n=DP^{-1}\circ Df^n\circ DP$, it follows that
$\sum_i\lambda_i^c(\tilde{f},x)>\sum_i\lambda_i^c(A)$ for any $x\in \widetilde{Z}$.

For each $q\in {\N}-\{0\}$ we define the set
$$
Z_q=\left\{x\in \widetilde{Z};\sum_i\lambda_i^c(\tilde{f},x)>\sum_i\lambda_i^c(A)+\log\left(1+\frac{1}{q}\right)\right\}.
$$
We have $\bigcup_{q=1}^{\infty}Z_q=\widetilde{Z}$, thus there is  $q$ such that $m(Z_q)>0$. For each $x\in Z_q$ follows that
$$
\displaystyle\lim_{n\rightarrow\infty}\frac{1}{n}\log|\jac \tilde{f}^n(x)|_{\widetilde{E}^c}|>\sum_i\lambda_i^c(A)+\log\left(1+\frac{1}{q}\right).
$$
So there is  $n_0$ such that for   $n\geq n_0$  we have
\begin{align*}
\frac{1}{n}\log|\jac \tilde{f}^n(x)|_{\widetilde{E}^c}|&>\sum_i\lambda_i^c(A)+\log\left(1+\frac{1}{q}\right)\\
                                         &>\frac{1}{n}\log e^{n\sum_i\lambda_i^c(A)}+\frac{1}{n}\log\left(1+\frac{1}{q}\right)^{n},
\end{align*}
it implies that
$$
|\jac \tilde{f}^n(x)|_{\widetilde{E}^c}|>\left(1+\frac{1}{q}\right)^{n}e^{n\sum_i\lambda_i^c(A)}.
$$
By this fact, for every $n>0$ we define
$$
Z_{q,n}=\left\{x\in Z_q;|\jac \tilde{f}^k(x)|_{\widetilde{E}^c}|>\left(1+\frac{1}{q}\right)^{k}e^{k\sum_i\lambda_i^c(A)},\,\,\forall\,\,k\geq n \right\}.
$$
There is $N>0$ with $\vol(Z_{q,N})>0,$ for some integer $q > 0.$

 For each $x\in D$, consider $B_x\subset{\R}^d$ a foliated box of  $\widetilde{W}^c_f.$
By compactness there is finite cover $\{B_{x_i}\}_{i=1}^{j}$ covering $\overline{D}.$ Since $W^c_f$ is absolutely continuous and the covering map is smooth, then  $\widetilde{W}^c_f$ is absolutely continuous, thus there is some $i$ and $p\in B_{x_i}$
such that $\vol(B_{x_i}\cap \widetilde{W}^c_f(p)\cap Z_{q,N})>0$.

There is a set $\pi^{-1}(R)\subset \widetilde{W}^c_f(p)$, where $\pi^{-1}(R)$ is as in the Proposition \ref{Prop vol center}
containing $p$ and the edges of the  cube $R$ are enough large
such that \mbox{$\vol(\pi^{-1}(R)\cap Z_{q,N})>0.$} Choose a   number $\varepsilon > 0$ as in Proposition \ref{Prop vol center}, such that $(1 + \varepsilon)^{d_c} < \left(1 + \frac{1}{q}\right).$ By Proposition \ref{Prop vol center} we have

\begin{equation}\label{eq4}
\vol(\tilde{f}^n(\pi^{-1}(R)))\leq C_0(1+\varepsilon)^{n d_c} e^{n\sum\lambda_i^c(A)}\vol(R),
\end{equation}
for any $n \geq 1.$

On the other hand, let  $\alpha>0$ be such that $\vol(\pi^{-1}(R)\cap Z_{q,N})=\alpha \vol(\pi^{-1}(R)),$ for $n > N,$ we have
\begin{align}\label{eq5}
\vol(\tilde{f}^n(\pi^{-1}(R)))&=\displaystyle\int_{\pi^{-1}(R)}|\jac \tilde{f}^n(x)|_{\widetilde{E}^c}| d \vol \nonumber\\
                &\geq\displaystyle\int_{\pi^{-1}(R)\cap Z_{q,N}}|\jac \tilde{f}^n(x)|_{\widetilde{E}^c}| d \vol\nonumber\\
                &>\displaystyle\int_{\pi^{-1}(R)\cap Z_{q,N}}\left(1+\frac{1}{q}\right)^{n}e^{n\sum_i\lambda_i^c(A)}d \vol\nonumber\\
                &>\left(1+\frac{1}{q}\right)^{n}e^{n\sum_i\lambda_i^c(A)} \vol(\pi^{-1}(R)\cap Z_{q,N})\nonumber\\
                &>\left(1+\frac{1}{q}\right)^{n}e^{n\sum_i\lambda_i^c(A)} \alpha\vol(\pi^{-1}(R))\nonumber\\
                &>\left(1+\frac{1}{q}\right)^{n}e^{n\sum_i\lambda_i^c(A)} \alpha\vol(R).
\end{align}
The equations $(\ref{eq4})$ and $(\ref{eq5})$ given us a contradiction when  $n$ is enough large, thus proving the Theorem \ref{Teo E}.

\end{proof}

\section{Proof of Theorem C}

For the construction we need the following results.

\begin{proposition}[\cite{BB}]\label{babo} Let $(M,m)$ be a compact manifold endowed with a $C^r$ volume
	form, $r \geq 2.$ Let $f$  be a $C^1$ and $m-$preserving diffeomorphisms of $M,$  admitting a dominated
	partially hyperbolic splitting $TM= E^s
	\oplus E^c \oplus E^u.$ Then there are arbitrarily $C^1-$close
	and $m-$preserving perturbation $g$ of $f,$ such that
	
	$$\displaystyle \int_M \log(J^c_g(x))dm > \displaystyle \int_M \log(J^c_f(x))dm,$$
	where $J^c_f(x)$ is the absolute value of the determinant of $Df$  restricted to $E^c_f(x).$
\end{proposition}

\begin{remark}
	When $f$ is $C^r, r\geq 1,$ in  Proposition \ref{babo} above  the perturbation $g$ also can be taken $C^r.$
\end{remark}

\begin{theorem}[\cite{FisherPotrieSambarino2014dynamical}]\label{Teo PFS}
Let $f:{\T}^d\rightarrow{\T}^d$ be a partially hyperbolic diffeomorphism that is isotopic to a linear Anosov automorphism along a path of partially hyperbolic diffeomorphisms, then $f$ is dynamically coherent. Moreover, for each $x\in{\T}^d$, \mbox{$W^c_f(x)=h^{-1}(W^c_A(h(x)))$}, where $h$ is the semiconjugacy between $f$ and $A$.
\end{theorem}

\begin{proof}[Proof of Theorem C]
Consider a linear Anosov diffeomorphism of the ${\T}^4$ induced by the matrix:

$$A_n = \left[\begin{array}{cccc}
0 & 0 & 0 & -1\\
1 & 0 & 0 & 14\\
0 & 1 & 0 & -19\\
0 & 0 & 1 & 8
\end{array}\right].$$

The characteristic polynomial of $A_n$ is  $p(t)  = t^4 - 8t^3 + 19t^2 - 14t + 1$ and its eigenvalues are approximately
$\beta^s=0,08;\,\,\, \beta^c_1=1,2;\,\,\, \beta^c_2=2,3$ \,and \, $\beta^u=4,3$.

We use the Proposition \ref{babo}  to  get a $C^2-$ volume preserving Anosov diffeomorphism  $f=A \circ \phi^{cu},$ where $\phi^{cu}$ is a volume preserving perturbation that preserves $E^s_{A_n}$ such that
$$
\lambda_1^c(f)+\lambda_2^c(f)=\displaystyle \int_{{\T}^4} \log(J^c_f(x))dm > \displaystyle \int_{{\T}^4} \log(J^c_A)dm=\lambda_1^c(A)+\lambda_2^c(A),
$$
By Theorem \ref{Teo E}, the foliation $W^c_f$ is non absolutely continuous, thus the disintegration of the volume along center box foliated is not Lebesgue.
By conjugacy between $f$ and $A,$ the stable index of $f$ is equal to one of $A.$ Using the Pesin's formula, we get
$$
h_m(f)=h_m(f^{-1})=-\lambda^{s}(f)=-\lambda^{s}(A)=h_{Top}(A)=h_{Top}(f),
$$
then the volume $m$ is the maximal measure entropy of  $f.$ By uniqueness of the maximal entropy measure of $f$ it follows that $h_{\ast}m=m,$ it means that $h$ preserves $m.$

By Theorem \ref{Teo PFS}, $W^c_f(x)=h^{-1}(W^c_A(h(x)))$, hence the disintegration of the volume along $W^c_f$ induces, by $h$, a disintegration of the volume along $W^c_A$, hence the disintegration
along $W^c_f$ is not atomic, because the disintegration along $W^c_A$ is Lebesgue.
\end{proof}

\section{Proof of Theorem B}

For to begin the construction, for each $n \geq 1,$  we  consider $A_n: \mathbb{T}^4 \rightarrow \mathbb{T}^4$ the Anosov automorphism the $\mathbb{T}^4$ induced by the matrices

$$A_n = \left[\begin{array}{cccc}
0 & 1 & 0 & 0\\
0 & 0 & 1 & 0\\
0 & 0 & 0 & 1\\
-1 & 3n+2 & -4n-3 & n + 4
\end{array}\right].$$

Denote by $p_n(t)$ a characteristic polynomial of $A_n,$ we have
$$p_n(t)  = t^4 - (n + 4)t^3 + (4n + 3)t^2 - (3n + 2)t + 1.$$

\begin{proposition} For large $n,$ the matrix $A_n$ have four eigenvalues $0 < \beta^s_n < 1 < \beta^c_{1,n} < \beta^c_{2,n} < \beta^u_n,$ such that
\begin{equation*}
	\frac{\beta^u_n}{n} \rightarrow 1, \hspace{0.3cm} \beta^c_{1,n} \rightarrow 1^+, \hspace{0.3cm} \beta^c_{2,n} \rightarrow 3^- \hspace{0.3cm}\mbox{and}\hspace{0.3cm} 3n \cdot \beta^s_n \rightarrow 1.
	\end{equation*}
\end{proposition}

\begin{proof}
 We have $p_n(n) =  1 - 2n \; \mbox{and} \; p_n(n + 1)  = n^3 - n^2 -4n - 1,$
for $n$ enough large  $p_n$ has a root $\beta^u_n \in (n, n+ 1),$ so $\displaystyle\lim_{n \rightarrow +\infty} \frac{\beta^u_n}{n } = 1.$

Fix $\varepsilon >0$, we have
$p_n(1) = -1$ and  \mbox{$p_n( 1 +  \varepsilon)  =  (2 + \varepsilon - \varepsilon^2 )\varepsilon n + \varepsilon^4 -3\varepsilon^2 - 4\varepsilon - 1,
$}
if $\varepsilon $ is very small, when $n$ is large we have $p_n(1 +  \varepsilon ) > 0,$ so $p_n$ has a root $\beta^c_{1,n} \in (1, 1 + \varepsilon),$ for large values of $n.$

Again, fix $ \varepsilon >0$ a small arbitrary number, we have
$p_n(3) =-5$ and \\ \mbox{ $p_n( 3 - \varepsilon)  =  (6 - 5\varepsilon + \varepsilon^2)\varepsilon n -16\varepsilon + 21 \varepsilon^2 -8\varepsilon^3 + \varepsilon^4 - 5,$}
if $\varepsilon $ is very small, when $n$ is large we have $p_n(3 - \varepsilon ) > 0,$ so $p_n$ has a root $\beta^c_{2,n} \in (3 - \varepsilon, 3),$ for large values of $n.$

Since $p_n$ has four real roots,
$
\beta^u_n \cdot \beta^c_{1,n} \cdot \beta^c_{2,n}  \cdot \beta^s_n = 1,
$
 and
$$
\displaystyle\lim_{n \rightarrow +\infty} \left[\frac{n}{\beta^u_n} \cdot \frac{1}{\beta^c_{1,n}}\cdot \frac{3}{\beta^c_{2,n}}\right] = 1,
$$
then as $\beta^s_n = \frac{1}{\beta^u_n\cdot \beta^c_{1,n}\cdot \beta^c_{2,n}},$ finally we have $\displaystyle\lim_{n \rightarrow +\infty} 3n \beta^s_n = 1.$
\end{proof}

We go to denote $\beta^c_{1,n} = 1 + \alpha_n,$ such that $\alpha_n \rightarrow 0^+. $ 

\begin{proposition} Let $v^u_n, v^c_{2, n}, v^c_{1, n}, v^s_{n} $ be corresponding unit eigenvectors of $\beta^u_n$, $\beta^c_{2,n}$, $\beta^c_{1,n}$ and $\beta^s_n$ respectively. There are unit vectors $ v^u, v^c_{2}, v^c_{1}, v^s,$ such that $$v^u_n \rightarrow v^u, v^c_{2, n}\rightarrow v^c_{2}, v^c_{1, n}\rightarrow v^c_{1}, v^s_{n}\rightarrow v^s ,$$
moreover $\mathcal{B}  = \{v^u, v^c_{2}, v^c_{1}, v^s\}$ is a basis of $\mathbb{R}^4.$
\end{proposition}

\begin{proof}
The vectores\\
$w^s_{n}=\left(
           \begin{array}{c}
             1 \\
             \beta^s_{n} \\
             (\beta^s_{n})^2 \\
             (\beta^s_{n})^3 \\
           \end{array}
         \right)
$,
$w^c_{1,n}=\left(
           \begin{array}{c}
             1 \\
             \beta^c_{1,n} \\
             (\beta^c_{1,n})^2 \\
             (\beta^c_{1,n})^3 \\
           \end{array}
         \right)$,
$w^c_{2,n}=\left(
           \begin{array}{c}
             1 \\
             \beta^c_{2,n} \\
             (\beta^c_{2,n})^2 \\
             (\beta^c_{2,n})^3 \\
           \end{array}
         \right)
$ and
$w^u_{n}=\left(
           \begin{array}{c}
             1 \\
             \beta^u_{n} \\
             (\beta^u_{n})^2 \\
             (\beta^u_{n})^3 \\
           \end{array}
         \right)
$
 are eigenvectors of $A_n$ associated with the respective eigenvalues $\beta^s_{n}$, $\beta^c_{1,n}$, $\beta^c_{2,n}$ and $\beta^u_{n}$. When $n\rightarrow\infty$, we have
\begin{align*}
w^s_{n}=\left(
           \begin{array}{c}
             1 \\
             \beta^s_{n} \\
             (\beta^s_{n})^2 \\
             (\beta^s_{n})^3 \\
           \end{array}
         \right)
\,\,\rightarrow \,\,
\left(
  \begin{array}{c}
    1 \\
    0 \\
    0 \\
    0 \\
  \end{array}
\right)
, \,\,\,\,\,\,\,\,\,
w^c_{1,n}&=\left(
           \begin{array}{c}
             1 \\
             \beta^c_{1,n} \\
             (\beta^c_{1,n})^2 \\
             (\beta^c_{1,n})^3 \\
           \end{array}
         \right)
\,\,\rightarrow \,\,
\left(
  \begin{array}{c}
    1 \\
    1 \\
    1 \\
    1 \\
  \end{array}
\right)\\
w^c_{2,n}=\left(
           \begin{array}{c}
             1 \\
             \beta^c_{2,n} \\
             (\beta^c_{2,n})^2 \\
             (\beta^c_{2,n})^3 \\
           \end{array}
         \right)
\,\,\rightarrow \,\,
\left(
  \begin{array}{c}
    1 \\
    3 \\
    9 \\
    27 \\
  \end{array}
\right)
\,\,\,\,\, \mbox{and} \,\,\,\,\,
\frac{1}{n^3}w^u_{n}&=\left(
           \begin{array}{c}
             1/n^3 \\
             \beta^u_{n}/n^3 \\
             (\beta^u_{n})^2/n^3 \\
             (\beta^u_{n})^3/n^3 \\
           \end{array}
         \right)
\,\,\rightarrow \,\,
\left(
  \begin{array}{c}
    0 \\
    0 \\
    0 \\
    1 \\
  \end{array}
\right)
\end{align*}
which are LI vectors.
\end{proof}

Denote by $H^u_n = \langle v^u_n \rangle,$$ $$H^c_{1,n} = \langle v^c_{1,n} \rangle,$ $H^c_{2,n} = \langle v^c_{2,n} \rangle$ and $H^s_n = \langle v^s_{n} \rangle.$

We can see $A_n: \mathbb{T}^4 \rightarrow \mathbb{T}^4$ with two partially hyperbolic decomposition.
\begin{equation}\label{first}
TM =  F^u_n \oplus F^c_n \oplus F^s_n
\end{equation}
where $F^u_n = H^u_n \oplus H^c_{2,n} ,$  $F^c_n = H^c_{1,n}$ and $F^s_n= H^s_n$ and
\begin{equation}\label{second}
TM =  E^u_n \oplus E^c_n \oplus E^s_n
\end{equation}
where $E^u_n = H^u_n  ,$  $E^c_n = H^c_{1,n} \oplus H^c_{2,n}$ and $E^s_n= H^s_n.$

\begin{remark}\label{simple}
From now on, for simplicity we can consider $v^u = (1,0,0,0)$, \mbox{$v^c_2 = (0,1,0,0)$}, $v^c_1 = (0,0,1,0) $ and $v^s = (0,0,0,1).$
\end{remark}

\subsection{Perturbing $A_n$} \label{perturb}


By  Proposition \ref{babo}, consider $g_n$ a small perturbation of $A_n,$ an Anosov diffeomorphism such that $A_n$ and $g_n$ are close in the $C^1$ topology and

\begin{equation}\label{ineqcenter}
\displaystyle\int_{\mathbb{T}^4} \log (\jac g_n| E^c_{g_n}) dm > \displaystyle\int_{\mathbb{T}^4} \log (\jac A_n| E^c_{A_n}) dm.
\end{equation}

Also we go to require that

\begin{equation}
||Dg_n(\cdot)| F^c_{g_n}|| < 1 + 2\alpha _n,   ||Dg_n(\cdot)| E^c_{g_n}|| < 3 + 3\alpha _n  \;\; \mbox{and} \;\; H_n := A_n^{-1}\circ g_n \rightarrow_{C^1} Id.\label{convergeId1}
\end{equation}

It is possible by Proposition \ref{babo}, since $g_n = A_n \circ H_n$ and $g_n$ can be found  $C^1-$arbitrarely close to $A_n.$

By continuity of the partially hyperbolic decomposition, if $g_n$ is $C^1$ close to $A_n,$ then there is an invariant partially hyperbolic splitting  $T \mathbb{T}^4 = E^u_{g_n} \oplus E^c_{g_n} \oplus E^s_{g_n},$ where $\dim (E^{\sigma}_{g_n}) = \dim (E^{\sigma}_{n}), \sigma \in \{s,c,u\}$ and  $E^{\sigma}_{g_n}$ and  $E^{\sigma}_{n},  \sigma \in \{s, c, u\}$ are close.
Again, by continuity of partially hyperbolic splitting, each $g_n$ admits a splitting $T\mathbb{T}^4 = F^u_{g_n} \oplus F^c_{g_n} \oplus F^s_{g_n},$  close to $F^u_{n} \oplus F^c_{n} \oplus F^s_{n}.$

Now we go to modify the stable index of a fixed point of $g_n$. For this we use the arguments of a well-known lemma of
\cite[Lema 1.1]{Franks1971necessary}, which allows us to make a change in the differential of $g_n$
in a finite number of points. In \cite[Proposition 7.4]{BDP} shows that if the initial diffeomorphism is conservative then we can take the perturbation also conservative. There is a more general version of this result in
\cite[Lema 2.4]{LLS}. For our proposes we only need the following lemma.

\begin{lemma}[\cite{BDP}]\label{LemBDP}
For any $N\in{\N}$ e $\rho>0$, there is a $\delta>0$ and a neighborhood of the identity
$B(\delta, Id)\subset SL(N,{\R})$ such that for any $A\in B(\delta, Id)$ there is a
$h\in \diff^r_m({\R}^N)$ satisfying the properties:
\begin{itemize}
  \item $h=Id$ out of a unitary ball centered on $0$,
  \item $h(0)=0$ and $Dh(0)=A$,
  \item $||Dh-Id||<\rho$.
\end{itemize}
\end{lemma}

Let $p_n$ be a fixed point for $g_n.$ For each fixed $n,$ consider a system $\{V_{n,j}\}_{j = 0}^{+\infty}$ of  small open balls centered in $p_n.$ The neighborhoods $V_{n,j} $ are constructed after.

Fixed $n > 0,$ enough large, consider matrixes $D_n$ satisfying:

\begin{enumerate}
\item $D_n $ and $Dg_n(p_n)$ have the same eigenspaces and the corresponding eigenvalues have the same sign.
\item $D_n| F^u_{g_{n}}(p_n) =  Dg_{n}(p_n)| F^u_{g_{n}}(p_n),$
\item $|| D_n| F^c_{g_{n}}(p_n)|| = 1 - \alpha_n, \;\;$
\item $D_n| F^s_{g_{n}}(p_n)$ is taken coherently with $\det(D_n) = \det(Dg_n(p_n)).$
\end{enumerate}

Since $g_n$ can be taken $C^1-$arbitrarily close to $A_n,$ we can to choose $g_n$ such that $\angle(F^{\sigma}_{g_n}(\cdot), F^{\sigma}_{n}(\cdot) )\rightarrow 0 $  uniformly in $\mathbb{T}^4,$ for each $ \sigma \in\{s, c, u\}.$ For each $n > 0$ and $x \in \mathbb{T}^4$ it is possible to find a basis $\{w^u_n(x), w^c_{2,n}(x), w^c_{1,n}(x) w^s_n(x)\}$ composed by unit vectors of $T_x\mathbb{T}^4$ uniformly close to $\{v^u_n, v^c_{2,n}, v^c_{1,n} v^s_n\}. $  By uniform convergence, the  change matrix $C_n(x),$ from the basis $\mathcal{B}_n' = \{w^u_n(x), w^c_{2,n}(x), w^c_{1,n}(x), w^s_n(x)\}$ to the canonical basis,  is uniform close to the identity matrix (see Remark \ref{simple}).

The matrixes $D_n$ will be  diagonal matrixes, in the basis $\mathcal{B}_n'.$ Denote by $\theta^u_n,$ $\theta^c_{1,n},$ $\theta^c_{2,n}$ and  $\theta^s_n = \frac{1}{\theta^u_n\cdot \theta^c_{1,n} \cdot \theta^c_{2,n}},$ the four eigenvalues of $Dg_n(p_n)$ respectively in the directions of the basis the basis $\mathcal{B}_n'.$  Then, using the items above, in the basis $\mathcal{B}_n'$ we have:
\begin{align*}
	[Dg_{n}(p_n)]^{-1} D_n &=\left[\begin{array}{cccc}
		\theta^u_n & 0 & 0& 0\\
		0 & \theta^c_{2,n}& 0 & 0\\
		0 & 0 &  \theta^c_{1,n}& 0\\
		0 & 0 & 0 & \frac{1}{\theta^u_n\theta^c_{1,n}\theta^c_{2_n}}
	\end{array}\right]^{-1}  \left[\begin{array}{cccc}
		\theta^u_n & 0 & 0& 0\\
		0 & \theta^c_{2,n} & 0 & 0\\
		0 & 0 &  1 - \alpha_n& 0\\
		0 & 0 & 0 & \frac{1}{\theta^u_n\theta^c_{2,n} (1 - \alpha_n)}
	\end{array}\right]\\
	&=\left[\begin{array}{cccc}
		1 & 0 & 0 & 0\\
		0 & 1 & 0 & 0\\
		0 & 0 & \frac{1-\alpha_n}{\theta^c_{1,n}} & 0\\
		0& 0 & 0 & \frac{\theta^c_{1,n}}{1 - \alpha_n}
	\end{array}\right].
\end{align*}
 From $\ref{convergeId1},$ we have $1<\theta^c_{1,n}<1+2\alpha_n$  it leads that
 \mbox{$\max\left\{|1 - \frac{\theta^c_{1,n}}{1-\alpha_n}|, |1 - \frac{1-\alpha_n}{\theta^c_{1,n}}|\right\} < 3,5\alpha_n,$} for $n$ large. Since $C_n(x)$ is uniform close to the identity, then in the canonical basis  we have

 \begin{equation}
 ||[Dg_n(p_n)]^{-1} \cdot D_n - Id || < 4\alpha_n, \;\mbox{for large values of $n.$} \label{important}
 \end{equation}

 Fix $\delta > 0,$ such that every $h\in \diff^1_m(\mathbb{T}^4)$ that is $\delta-C^1-$close to $Id$ is homotopical to $Id.$

Since by $\ref{important}$ the norm  $||[Dg_n(p_n)]^{-1} \cdot D_n - Id || \rightarrow 0$  when $n \rightarrow + \infty. $

Now as the Lemma \ref{LemBDP}, fix $\sigma>0$, which depends only on $\delta$, such that for all $h\in \diff^r_m({\R}^4)$
with $||h-Id||_{C^1}\leq \sigma$ is homotopic to identity. For Equation \ref{important} take $n$ so that $4\cdot \alpha_n$ is enough small
such that by Lemma \ref{LemBDP} there is a  $h_n\in \diff^r_m({\R}^4)$  satisfying:
\begin{itemize}
	\item $||h_{n} - Id||_{C^1} < \sigma,$
	\item $h_{n} =Id$ in $p_n$ and out of a unitary ball centered on $p_n$.
	\item $ Dh_{n} (p_n) = [Dg_n(p_n)]^{-1} \cdot D_n.$
\end{itemize}

For each $j$ suppose that $V_{n,j}=B(p_n,\varepsilon_j)$ ($\varepsilon_j$ It will be defined later).
We define $h_{n,j}\in \diff^r_m({\R}^4)$ by
$$
h_{n,j}(x)=\varepsilon_j h_n(\frac{x}{\varepsilon_j}),\,\, \mbox{for\,\,all}\,\,x\in{\R^4}, \,\,\mbox{then}
$$
\begin{itemize}
    \item $||h_{n,j} - Id||_{C^1} < \sigma,$
	\item $h_{n,j}(p_n) = p_n, \; h_{n,j} = Id \; \mbox{out of}\,\, V_{n,j},$
	\item $Dh_{n,j} (p_n)=Dh_{n} (p_n)= [Dg_n(p_n)]^{-1} \cdot D_n.$
\end{itemize}
Let $g_{n,j}=g_n\circ h_{n,j}$. Then
$$
Dg_{n,j}=Dg_{n}(p_n)\circ Dh_{n,j}(p_n)= Dg_{p_n}\circ [Dg_n(p_n)]^{-1} \circ D_n=D_n
$$

\begin{figure}[!htb]
\centering
\includegraphics[scale=0.95]{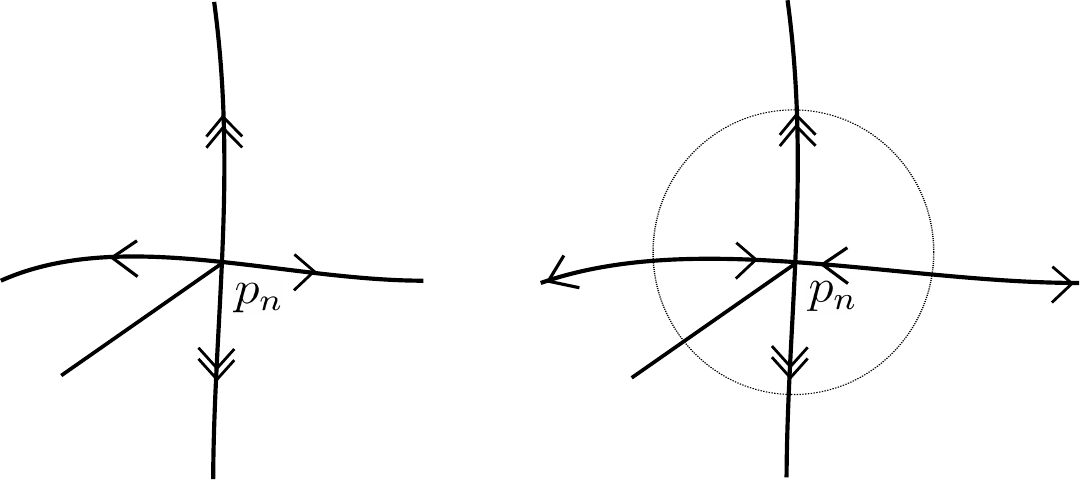}
\caption{The fixed point $p_n$ of $g_n$ and $g_{n,j}$, respectively}\label{PerturbCeter1}
\end{figure}

\begin{remark}
We observe that if $g_{n}$ is $C^r, r \geq 1,$ then each $g_{n,j}$ can be taken $C^r.$  That is because the perturbation of $g_n$
involves exponential function and some bump functions which are $C^{\infty}.$
\end{remark}

\subsection{The choice of the open system $\{V_{n,j}\}_{j = 0}^{+\infty}$}\label{neigh}
Fix $n >0$ and $\varepsilon> 0 $ a small number. Let $V_{n,0} = B(p_n, \varepsilon)$ be the open ball centered in $p_n$ with radius $\varepsilon.$ Since $p_n$ is a fixed point for the diffeomorphism $g_n,$ then the distance between $p_n$ and $g_n^j(\mathbb{T}^4 \setminus V_{n,0})$ is positive for any $j$ a integer number.  Define recursively $\varepsilon_j = (20n)^{-j} \varepsilon_{j -1}, j = 1, 2, \ldots$ and $V_{n,j} = B(p_n, \varepsilon_j), j = 1, 2, \ldots.$ Since the Lipschitz constants of $g_{n,j}$ and its inverse is bounded by $20n,$ we have $g_{n,j}^k(x) \notin V_{n,j}$ for every $x \notin V_{n , j-1},$ $|k| \leq j$ and $j =1,2,\ldots.$

\subsection{Properties of the diffeomorphisms  $g_{n,j}$}

\begin{lemma}\label{Lem Prop 1}
The diffeomorphisms $g_{n,j}$ are partially hyperbolic diffeomorphisms homotopic to $A_n$.
\end{lemma}

\begin{proof}
It is because $g_{n,j} = g_n \circ h_{n,j}=A_n\circ H_n\circ h_{n,j}$ and  $H_n\circ h_{n,j}$
is arbitrarily $C^1-$close of the identity, for $n$ enough large. The partial hyperbolicity follows from Lemma \ref{lema Adp Ponce-Tahzibi}.
\end{proof}

\begin{lemma}\label{convergence} For each large $n$ the diffeomorphisms $g_{n,j}$  are (absolute) partially hyperbolic
and there are two decompositions $F_{g_{n,j}}^u \oplus F_{g_{n,j}}^c \oplus F_{g_{n,j}}^s$ and
$E_{g_{n,j}}^u \oplus E_{g_{n,j}}^c \oplus E_{g_{n,j}}^s$ such that
\begin{enumerate}
\item $E^{\sigma}_{g_{n,j}}(x) \rightarrow  E^{\sigma}_{g_{n}}(x), \sigma \in \{s,c,u\},  \; \mbox{where} \; j \rightarrow + \infty,$
\item $F^{\sigma}_{g_{n,j}}(x) \rightarrow  F^{\sigma}_{g_{n}}(x), \sigma \in \{s,c,u\},  \; \mbox{where} \; j \rightarrow + \infty,$
\end{enumerate}
\end{lemma}

\begin{proof}
The existence of two partially hyperbolic decompositions follow, as above, from the Lemma \ref{lema Adp Ponce-Tahzibi}. In fact, for both
decompositions we can take $\varepsilon$ (depends only on $\Theta$) constant, because $\Theta>0$ for $n$ great enough, moreover
$H_n\circ h_{n,j}$ is arbitrarily $C^1-$close of the identity.


Now fix the decomposition $E_{g_{n,j}}^u \oplus E_{g_{n,j}}^c \oplus E_{g_{n,j}}^s,$ for this decompositions the are the cone fields (constructed in the Lemma \ref{lema Adp Ponce-Tahzibi}).

By construction of the family $\{V_{n,j}\}$, we have
$\diam (V_{n,j})\rightarrow 0$ where $j\rightarrow \infty$ and for any $x\neq p_n$,
there is $j_x\geq 0$ such that

\begin{itemize}
  \item[(1)] $x \notin \overline{V_{n,j}}$ for all $j\geq j_x$
  \item[(2)] $g^k_{n,j}(x)=g^k_n(x)$ for all $-j< k< j$ and $j>j_x$
  \item[(3)] $g_{n,j}=g_n$ out of $V_{n,j}$ for all $j>0$
\end{itemize}

For $\sigma=cu$, suppose that $C^{cu}(x,\beta)$ is the cone field in the direction $E^{cu}$, then
$$
E^{cu}_{g_n}(x)=\displaystyle\lim_{k\rightarrow +\infty}
Dg^k_j\left(g^{-k}_n(x)\right)\left(C^{cu}\left(g^{-k}_n(x),\beta\right)\right)
$$
and
$$
\displaystyle\lim_{j\rightarrow +\infty} E^{cu}_{g_{n,j}}(x)=
\displaystyle\lim_{j\rightarrow +\infty}\left(\displaystyle\lim_{k\rightarrow +\infty}Dg^k_{n,j}
\left(g^{-k}_{n,j}(x)\right)\left(C^{cu}\left(g^{-k}_{n,j}(x),\beta\right)\right)\right)
$$

By  cone construction, the sequence of cones $Dg^k_{n,j}(g^{-k}_{n,j}(x))(C^{cu}(g^{-k}_{n,j}(x),\beta))$
is decreasing nested for $k=1,2,3,\ldots$. If $j>j_x$ we have,  by definition of $j_x$, that
$$
E^{cu}_{g_{n,j}}(x)\subset Dg^j_{n,j}
\left(g^{-j}_{n,j}(x)\right)\left(C^{cu}\left(g^{-j}_{n,j}(x),\beta\right)\right)=
Dg^j_{n}
\left(g^{-j}_{n}(x)\right)\left(C^{cu}\left(g^{-j}_{n}(x),\beta\right)\right).
$$

Since
$Dg^j_{n} \left(g^{-j}_{n}(x)\right)\left(C^{cu}\left(g^{-j}_{n}(x),\beta\right)\right) \rightarrow E^{cu}_{g_n}$
where $j\rightarrow \infty$, then
$$
E^{cu}_{g_{n,j}}(x)\rightarrow E^{cu}_{g_n}(x).
$$
Moreover by construction $E^{cu}_{g_{n,j}}(p_n)=E^{cu}_{g_{n}}(p_n)$.

Analogously, we proof for the other directions $E_{g_{n,j}}^{cs}, E_{g_{n,j}}^{u}, E_{g_{n,j}}^{s}$  and also for the other decomposition
$F_{g_{n,j}}^u \oplus F_{g_{n,j}}^c \oplus F_{g_{n,j}}^s$.
\end{proof}

\begin{lemma}\label{Lem Prop 3}
For $n$ and $j$ enough large and for both partially hyperbolic decompositions of $g_{n,j}$ holds:
\begin{itemize}
 \item[(1)] There is $\alpha>0$ such that  $\angle(E^{\sigma}_{g_{n,j}}, E^{\sigma}_{A_n})>\alpha$ and $\angle(F^{\sigma}_{g_{n,j}}, F^{\sigma}_{A_n})>\alpha,$
 \item[(2)] The foliations $W^s(g_{n,j})$ and $W^u(g_{n,j})$ are quasi-isometric in the universal cover,
 \item[(3)] $g_{n,j}$ is dynamically coherent and the foliation $W^c(g_{n,j})$ is quasi-isometric in the universal cover.
  \item[(4)] If  $h$ is the semiconjugacy between $A_n$  and $g_{n,j}$, then $W^c_{g_{n,j}}(x)=h^{-1}(W^c_{A_n}(h(x)))$.
In particular the two dimensional foliation  $W^{c}_{g_{n,j}}$ is not compact.

\end{itemize}
\end{lemma}

\begin{proof}
The item (1) follows of the Lemma \ref{convergence}. The item (2) follows of the item (1) and Proposition \ref{B1}.
The item (2) together with
the Theorem \ref{B2} implies that $g_{n,j}$ is dynamically coherent and again by item (1) and
Proposition \ref{B1} implies that $W^c(g_{n,j})$ is quasi-isometric in the universal cover. The item (4) follows of the Theorem \ref{Teo PFS}. In fact, since $H_n \circ h_{n,j} $ is arbitrarily close to $Id$ in the $C^1-$topololy, its combined with the Lemma 6.2 of the appendix, leads us to each $A_n \circ I_t, t \in [0,1],$ is partially hyperpolic, where $I_t = (1-t)Id + t(H_n \circ h_{n,j}).$ It is clear, because if $H_n \circ h_{n,j}$ is $C^1-$close to $Id,$ the $I_t,$ so are. The path $t \mapsto A_n \circ I_t$ is fully contained in the partially hyperbolic path connect component of $A_n,$ we are able to apply Theorem \ref{Teo PFS}.
\end{proof}

\subsection{Conclusion of the proof of Theorem C}

Fix a large $n$ and $x \neq p_n.$ Since $E^c_{g_{n,j}}(x) \rightarrow E^{c}_{g_n}(x) $ and $g_{n,j} = g_{n}$ out of $V_{n,j},$  thus  for every $j > j_x$ we have \mbox{$Dg_{n,j}(x) =   Dg_{n}(x),$}   then by Lemma \ref{convergence}

$$\jac g_{n,j}(x)| E^c_{g_{n,j}} \rightarrow \jac g_n(x)| E^c_{g_n}(x).$$

Since $H_n$ and $(H_n \circ h_{n,j})$ are $C^1$ close to  $Id$  and $g_{n,j} = A_n \circ H_n \circ h_{n,j},$ we have $||Dg_{n,j}|| < 20n$ for large $n,$ then by dominated convergence:

\begin{equation}\label{last}
\int_{\mathbb{T}^4 } \log(\jac g_{n,j}(x)| E^c_{g_{n,j}})dm \rightarrow \int_{\mathbb{T}^4}\log(\jac g_{n}(x)| E^c_{g_n}) dm >  \int_{\mathbb{T}^4}\log( \jac A_{n}(x)| E^c_{n})dm.
\end{equation}

The DA-diffeomorphisms $g_{n,j}$ are not Anosov. In fact, if $g_{n,j}$ was Anosov, then they would be conjugated to $A_n,$ in particular the stable index of fixed points of $A_n$ and $g_{n,j}$ would be the same.

Take a volume  preserving partially hyperbolic diffeomorphism $g_{n,j}$ satisfying

\begin{equation*}
\int_{\mathbb{T}^4 } \log(\jac g_{n,j}(x)| E^c_{g_{n,j}})dm  >  \int_{\mathbb{T}^4}\log(\jac A_{n}(x)| E^c_{n})dm.
\end{equation*}

it is possible by the expression $(\ref{last}).$

Since $g_{n,j}$ is not Anosov, we have $g_{n,j} \in \partial (\overline{\mathcal{A}(\mathbb{T}^4)}) \cup DA^r_m(\mathbb{T}^4) \setminus \overline{\mathcal{A}(\mathbb{T}^4)},$ in both cases it is possible to perturb $g_{n,j}$ to an $f_n \in DA^r_m(\mathbb{T}^4) \setminus \overline{\mathcal{A}(\mathbb{T}^4)}, r \geq 2,$ an stably ergodic partially hyperbolic diffeomorphism, such that
\begin{equation*}
\int_{\mathbb{T}^4 } \log(\jac f_{n}(x)| E^c_{f_{n}})dm  >  \int_{\mathbb{T}^4}\log(\jac A_{n}(x)| E^c_{n})dm.
\end{equation*}


The existence of such $f_n$  is due to the   continuity of the structure
\mbox{$F^u \oplus F^c \oplus F^s,$}
$\dim F^c = 1$ and  Theorems A and B in \cite{HHU}. Now consider \mbox{$U \subset DA^r_m(\mathbb{T}^4) \setminus \overline{\mathcal{A}(\mathbb{T}^4)}$}, $r \geq 2, $ an small $C^1-$neighborhood of $f_n.$ If we take $U$ a suitable neighborhood, $f_n,$ such that any $f \in U$ is ergodic, isotopic to $A_n,$ and

\begin{equation}\label{finalmente2}
\lambda^c_1(f) + \lambda^c_2(f)  = \int_{\mathbb{T}^4 } \log(\jac f(x)| E^c_{f})dm  > \log(\beta^c_{1,n} ) + \log(\beta^c_{2,n} ).
\end{equation}

Equation $(\ref{finalmente2}),$  joint with Theorem B implies that the bi-dimensional center foliation tangent to $E^c_f$ is not absolutely continuous.


\section{Appendix: Cone Conditions}

In this section we adapt a result of appendix of  \cite{PT}
from ${\T}^3$ to ${\T}^d$ using all the notation.
As in \cite{PT} we remember the equivalent definition of absolutely partially hyperbolic diffeomorphism by discretion of the cone.

\begin{definition}
Given an orthogonal splitting of the tangent bundle of $M$, $TM=E\oplus F$, and a constat $\beta>0$,
for any $x\in M$ we define the cone centered in $E(x)$ with angle $\beta$ as
$$
C(E,x,\beta)=\{v\in T_xM:\, ||v_F||\leq\beta ||v_E||,\, \mbox{onde}\,v=v_E+v_F, v_E\in E(x), v_F\in F(x)  \}.
$$
\end{definition}

Given a partially hyperbolic diffeomorphism   $f:{\T}^d\rightarrow{\T}^d$ with invariant splitting $TM=E^s\oplus E^c\oplus E^u$,
there is an adapted inner product (and then an adapted norm) with respect to which the splitting is orthogonal (see \cite{pesin2004lectures}).
Thus, given $\beta>0$ we can define standard families of cones centered  on the fiber bundles $E^{\sigma}(x)$ with angle $\beta>0$,
$C^{\sigma}(x,\beta)$, $\sigma= s,c,u,cs,cu$.

Consider $f:M\rightarrow M$ an absolutely partially hyperbolic diffeomorphism. Using an adapted norm
$||\cdot||$, we can consider the invariant splitting \mbox{$TM=E^s\oplus E^c\oplus E^u$}
As being an orthogonal splitting, and there are numbers
$$
0<\lambda_1\leq \mu_1<\lambda_2\leq\mu_2<\lambda_3\leq\mu_3,\,\,\mu_1<1,\,\,\lambda_3>1
$$
for which
\begin{align*}
\lambda_1\leq||Df(x)|E^s(x)||\leq\mu_1,\\
\lambda_2\leq||Df(x)|E^c(x)||\leq\mu_2,\\
\lambda_3\leq||Df(x)|E^u(x)||\leq\mu_3.
\end{align*}

Partial hyperbolicity can be described in terms of invariant cone families (see \cite{pesin2004lectures}, pg. 15).
More specifically, let $f:{\T}^d\rightarrow{\T}^d$ be a partially hyperbolic diffeomorphism and let
$$
T_x{\T}^d=E^s(x)\oplus E^c(x)\oplus E^u(x)
$$
a orthogonal splitting of $T{\T}^d$. Give $\beta>0$ define the families of cones

\begin{align*}
&C^s(x,\beta)=C(x,E^s(x), \beta),\,\,\,\,\,\,\,\,\,\,\,\,\,\,\,C^u(x,\beta)=C(x,E^u(x), \beta),\\
&C^{cs}(x,\beta)=C(x,E^{cs}(x), \beta),\,\,\,\,\,\,\,\,\,\,\,C^{cu}(x,\beta)=C(x,E^{cu}(x), \beta),
\end{align*}

where $E^{cs}(x)=E^c\oplus E^s(x),\mbox{e} E^{cu}(x)=E^c\oplus E^u(x)$. Then $f$ is absolutely partially hyperbolic if, and only if, there is $0<\beta<1$ and constants
$0< \mu_1<\lambda_2\leq\mu_2<\lambda_3$ with $\mu_1<1$ and $\lambda_3>1$ such that
\begin{align}\label{eq C1}
Df^{-1}(x)(C^{\sigma}(x,\beta))&\subset C^{\sigma}(f^{-1}(x),\beta),\,\, \sigma=s,cs, \nonumber\\
Df(x)(C^{\psi}(x,\beta))&\subset C^{\psi}(f(x),\beta),\,\, \psi=u,cu
\end{align}
and
\begin{align}\label{eq C2}
||Df^{-1}(x)v||&>\mu_1^{-1}||v||,\,\, v\in C^s(x,\beta), \nonumber\\
||Df^{-1}(x)v||&>\mu_2^{-1}||v||,\,\, v\in C^{cs}(x,\beta), \nonumber\\
||Df(x)v||&>\lambda_3||v||,\,\, v\in C^{u}(x,\beta), \\
||Df(x)v||&>\lambda_2||v||,\,\, v\in C^{cu}(x,\beta). \nonumber
\end{align}
For the linear case we find an explicit relation between the angle of the invariant cones families and the ratio of domination
between unstable, stable and central bundles.

Consider $f:{\T}^d\rightarrow{\T}^d$ a linear partially hyperbolic diffeomorphism and denote by
$\lambda_1^s,\lambda_2^s,\ldots,\lambda_i^s,\lambda_1^c\ldots,\lambda_j^c,\lambda_1^u,\ldots,\lambda_{k-1}^u,\lambda_k^u,$
its eigenvalues, where
$$
|\lambda_1^s|<|\lambda_2^s|<\ldots<|\lambda_i^s|<|\lambda_1^c|<\ldots<|\lambda_j^c|<|\lambda_1^u|<\ldots<|\lambda_{k-1}^u|<|\lambda_k^u|\,\,\,
\mbox{with}\,\,\, |\lambda_i^s|<1<|\lambda_1^u|
$$
Consider
$$
\Theta:=\min\left\{\frac{\lambda_1^c}{\lambda^s_i},\frac{\lambda^u_1}{\lambda^c_j}\right\}.
$$
We can choose a constant $\beta>0$ such that
\begin{equation}\label{eq C3}
1<(1+\beta)^2<\Theta.
\end{equation}
By definition of $\beta$ we have
$$
(1+\beta)|\lambda_i^s|<\frac{|\lambda_1^c|}{1+\beta}<(1+\beta)|\lambda_j^c|<\frac{|\lambda_1^u|}{1+\beta}
$$
We can find constants $\mu_1,\lambda_2,\mu_2,\lambda_3$ such that
\begin{equation}\label{eq C4}
(1+\beta)|\lambda_i^s|<\mu_1<\lambda_2<\frac{|\lambda_1^c|}{1+\beta}<(1+\beta)|\lambda_j^c|<\mu_2<\lambda_3<\frac{|\lambda_1^u|}{1+\beta}
\end{equation}
where ${\mu_1<1<\lambda_3}$

Now, we go to check that with the constants defined by \ref{eq C3} and \ref{eq C4}, the families of stable, unstable, center-stable and center-unstable cones satisfies \ref{eq C1} and \ref{eq C2}.

$\bullet$ Let $v=v_{s}+v_{cu}\in C^{cu}(x,\beta)$, then
$$
||Df(x)v_{s}||\leq|\lambda_i^s|||v_{s}||\leq |\lambda_i^s|\beta||v_{cu}||<\beta|\lambda_1^c|||v_{cu}||\leq\beta||Df(x)v_{cu}||
$$
This is $Df(x)(C^{cu}(x,\beta))\subset C^{cu}(f(x),\beta)$.\\

Furthermore
$$
||Df(x)v||^2\geq||Df(x)v_{cu}||^2\geq(\lambda_1^c)^2||v_{cu}||^2.
$$
By \ref{eq C4} we have that $|\lambda_1^c|>(1+\beta)\lambda_2$. Hence,
\begin{align*}
||Df(x)v||^2>(1+\beta)^2(\lambda_2)^2||v_{cu}||^2&\geq(\lambda_2)^2(||v_{cu}||^2+\beta^2||v_{cu}||^2)\\
                                                 &\geq (\lambda_2)^2(||v_{cu}||^2+||v_{s}||^2)=(\lambda_2||v||)^2
\end{align*}
Then $||Df(x)v||>\lambda_2||v||$.

Similarly we do to $C^{cs}(x,\beta)$, $C^{u}(x,\beta)$ and $C^{s}(x,\beta)$.

 The next lemma is completely analogous to (Lemma 6.1, \cite{PT}), as well as its demonstration. This lemma show the size of the $C^1$-neighborhood of the linear partially hyperbolic diffeomorphism such that all diffeomorphism in this neighborhood are absolutely partially hyperbolic diffeomorphism.

\begin{lemma}\label{lema Adp Ponce-Tahzibi}
Let $f:{\T}^d\rightarrow{\T}^d$ be a linear partially hyperbolic diffeomorphism with eigenvalues $\lambda_1^s,\lambda_2^s,\ldots,\lambda_i^s,\lambda_1^c\ldots,\lambda_j^c,\lambda_1^u,\ldots,\lambda_{k-1}^u,\lambda_k^u,$ where
$$
|\lambda_1^s|<|\lambda_2^s|<\ldots<|\lambda_i^s|<|\lambda_1^c|<\ldots<|\lambda_j^c|<|\lambda_1^u|<\ldots<|\lambda_{k-1}^u|<|\lambda_k^u|\,\,\,
\mbox{com}\,\,\, |\lambda_i^s|<1<|\lambda_1^u|,
$$
then there is a constant $\varepsilon>0$ such that, for any diffeomorphism $g:{\T}^d\rightarrow{\T}^d$ with $||g-Id||_{C^1}<\varepsilon$
(adapted norm of $f$), the composition $f\circ g$ is an absolutely partially hyperbolic diffeomorphism.
The constant $\varepsilon$ depends only on
$\Theta:=\min\left\{\frac{\lambda_1^c}{\lambda^s_i},\frac{\lambda^u_1}{\lambda^c_j}\right\}$.
\end{lemma}

\begin{proof}
Let $f$ be as in the statement and denote the invariant splitting of $f$ by $T_xM=E^s(x)\oplus E^c(x)\oplus E^u(x)$.
Consider  the adapted norm $||\cdot||$ in a way that the invariant splitting is orthogonal. Now, since $f$ is a linear
partially hyperbolic diffeomorphism we can choose constants
$$
0<\lambda_1\leq \mu_1<\lambda_2\leq\mu_2<\lambda_3\leq\mu_3,\,\,\mu_1<1,\,\,\lambda_3>1
$$
and let $\beta$ as defined in the equations \ref{eq C3} e \ref{eq C4}.

For $v\in T_xM$ we can write $v=v_s+v_c+v_u$ with $v_{\sigma}\in E^{\sigma}(x), \sigma=s,c,u.$\\

$\bullet$ If $v\in C^u(x,\beta)$, then $||v_{cs}||\leq\beta||v_{u}||$, where $v_{cs}=v_{s}+v_{c}$, thus
$$
||Df(x)v_{cs}||\leq\mu_2||v_{cs}||\leq\mu_2\beta||v_{u}||\leq\mu_2\beta(\lambda_3)^{-1}||Df(x)v_{u}||,
$$
this is, $Df(x)(C^u(x,\beta))\subset C^u(f(x),(\mu_2/\lambda_3)\beta)$.

$\bullet$ If $v\in C^{cu}(x,\beta)$, then $||v_{s}||\leq\beta||v_{cu}||$, where $v_{cu}=v_{c}+v_{u}$, thus
$$
||Df(x)v_{s}||\leq\mu_1||v_{s}||\leq\mu_1\beta||v_{cu}||\leq\mu_1\beta(\lambda_2)^{-1}||Df(x)v_{cu}||,
$$
this is, $Df(x)(C^{cu}(x,\beta))\subset C^u(f(x),(\mu_1/\lambda_2)\beta)$.

$\bullet$ If $v\in C^s(x,\beta)$, then $||v_{cu}||\leq\beta||v_{s}||$, where $v_{cu}=v_{c}+v_{u}$, thus
$$
||Df^{-1}(x)v_{cu}||\leq\lambda_2^{-1}||Df\circ Df^{-1}(x) v_{cu}||=\lambda_2^{-1}||v_{cu}||\leq\lambda_2^{-1}\beta||v_{s}||\leq\lambda_2^{-1}\beta\mu_1||Df^{-1}(x)v_{s}||,
$$
this is, $Df^{-1}(x)(C^s(x,\beta))\subset C^s(f^{-1}(x),(\mu_1/\lambda_2)\beta)$.

$\bullet$ If $v\in C^{cs}(x,\beta)$, then $||v_{u}||\leq\beta||v_{cs}||$, where $v_{cs}=v_{c}+v_{s}$, thus
$$
||Df^{-1}(x)v_{u}||\leq\lambda_3^{-1}||Df\circ Df^{-1}(x) v_{u}||=\lambda_3^{-1}||v_{u}||\leq\lambda_3^{-1}\beta||v_{cs}||\leq\lambda_3^{-1}\beta\mu_2||Df^{-1}(x)v_{cs}||,
$$
this is, $Df^{-1}(x)(C^{cs}(x,\beta))\subset C^{cs}(f^{-1}(x),(\mu_2/\lambda_3)\beta)$.

Define $\gamma:=\max\left\{\frac{\mu_2}{\lambda_3},\frac{\mu_1}{\lambda_2}\right\}<1$. So we have
\begin{align*}
Df(x)(C^{\sigma}(x,\beta))&\subset C^{\sigma}(f(x),\gamma\beta),\,\,\sigma=u,cu,\\
Df^{-1}(x)(C^{\psi}(x,\beta))&\subset C^{\psi}(f^{-1}(x),\gamma\beta),\,\,\psi=s,cs.
\end{align*}
Observe that by equations \ref{eq C3} and \ref{eq C4}, $\beta$ and $\gamma$ depends only on the rations
$\frac{\lambda_1^c}{\lambda^s_i}$ and $\frac{\lambda^u_j}{\lambda^c_1}$.

Since the invariant splitting is constant, we can take an $\varepsilon>0$ depending only on the ratios
$\frac{\lambda_1^c}{\lambda^s_i}$ and $\frac{\lambda^u_j}{\lambda^c_1}$ such that if $||g-Id||_0<\varepsilon$, then
\begin{align*}
Dg(x)(C^{\sigma}(x,\gamma\beta))&\subset C^{\sigma}(g(x),\beta),\,\,\sigma=u,cu,\\
Dg^{-1}(x)(C^{\psi}(x,\gamma\beta))&\subset C^{\psi}(g^{-1}(x),\beta),\,\,\psi=s,cs
\end{align*}
let
$$
L<\frac{1}{\mu_1},\,\,l>\frac{1}{\lambda_3},\,\,\frac{l}{L}>\gamma,\,\,
$$
with
$$
l||v||\leq||Dg(x)v||\leq L||v||.
$$
Thus we have (see Figure \ref{Cones2})
\begin{align*}
D(f\circ g)(x)(C^{\sigma}(x,\gamma\beta))&\subset C^{\sigma}(f(g(x)),\gamma\beta),\,\,\sigma=u,cu,\\
D(f\circ g)^{-1}(x)(C^{\psi}(x,\beta))&\subset C^{\psi}(g^{-1}(f^{-1}(x)),\beta),\,\,\psi=s,cs
\end{align*}
\begin{figure}[!htb]
\centering
\includegraphics[scale=0.8]{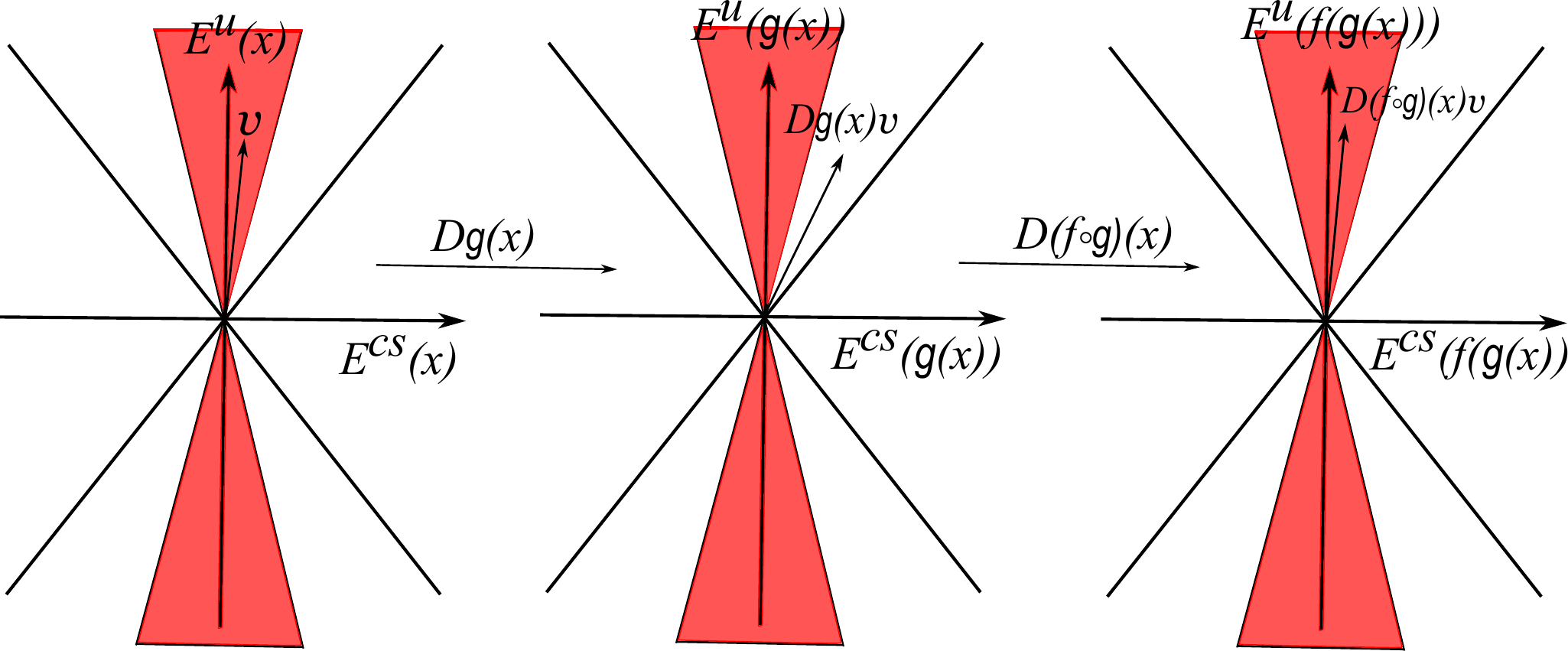}
\caption{}\label{Cones2}
\end{figure}

Now, we need to show uniform contraction and expansion on these families of cones.\\

$\bullet$ If $v\in C^u(x,\gamma\beta)$, then
$$
||D(f\circ g)(x)v||\geq\lambda_3||Dg(x)v||\geq\lambda_3l||v||.
$$
$\bullet$ If $v\in C^{cs}(x,\beta)$, then
$$
||D(f\circ g)^{-1}(x)v||\geq L^{-1}||Df^{-1}(x)v||>L^{-1}\mu_2^{-1}||v||.
$$
$\bullet$ If $v\in C^{cu}(x,\gamma\beta)$, then
$$
||D(f\circ g)(x)v||\geq \lambda_2||Dg(x)v||\geq \lambda_2l||v||.
$$
$\bullet$ If $v\in C^{s}(x,\beta)$, then
$$
||D(f\circ g)^{-1}(x)v||\geq L^{-1}||Df^{-1}(x)v||>L^{-1}\mu_1^{-1}||v||.
$$
Furthermore,
$$
0<L\mu_1<l\lambda_2\leq L\mu_2<l\lambda_3,\,\, \mbox{with}\,\, L\mu_1<1\,\,,\,\,l\lambda_3>1,
$$
so that $f\circ g$ is absolutely partially hyperbolic as we claimed.
\end{proof}

%
%
%

\end{document}